\documentclass[preprint,12pt]{elsarticle_mod}
\usepackage{graphicx}
\usepackage{amssymb}
\usepackage{amsthm}

\usepackage{amsmath}
\usepackage{amscd}

\theoremstyle{plain}
\newtheorem{theorem}[subsection]{Theorem}
\newtheorem{proposition}[subsection]{Proposition}
\newtheorem{lemma}[subsection]{Lemma}
\newtheorem{corollary}[subsection]{Corollary}

\newcommand{\C}{\mathbb C}

\newcommand{\R}{\mathbb{R}}
 
\journal{\ }

\begin{document}

\begin{frontmatter}

\title{Singular integrals and maximal functions:\\ the disk multiplier revisited}

\author[lbl]{Antonio C\'ordoba}
\ead{antonio.cordoba@uam.es}
\fntext[lbl]{The author is partially supported by the grant
MTM2011-22851 from the Ministerio de Ciencia e Innovaci\'{o}n (Spain).}

\address{ICMAT\\Departamento de Matem\'{a}ticas\\
Universidad Aut\'{o}no\-ma de Madrid
\\
28049 Madrid. Spain}

\begin{abstract}
Several estimates for singular integrals, maximal functions and the spherical summation operator are given in the spaces $L^p_{\text{\rm rad}}L^2_{\text{\rm ang}}(\R^n)$, $n\geq 2$. 
\end{abstract}


\end{frontmatter}

\section{Introduction}

A well-known
 open problem in Fourier analysis is the Bochner-Riesz operator conjecture, which asserts the $L^p$-boundedness of the Fourier multipliers
\[
 \widehat{T_\alpha f}(\xi)=\big(1-|\xi|^2\big)_+^\alpha \widehat{f}(\xi)
\]
on $L^p(\R^n)$, so long as 
\[
 \frac{2n}{n+1+2\alpha}
<p< \frac{2n}{n-1-2\alpha}
\]
where $0<\alpha<(n-1)/2$ and
\[
 \widehat{f}(\xi)=\int_{\R^n}e^{-2\pi i x\cdot\xi}f(x)\; dx
\]
denotes the Fourier transform in $\R^n$. 

The problem is well understood in dimensions $n=1$ and $n=2$ (see \cite{1}, \cite{2}, \cite{3}, \cite{4}). But in higher dimensions, although there are several interesting results by many authors, it remains open.

Its relevance is due, on the one hand, to the very natural question being asked, but also because of its close connection with some other basic objects, namely the  so-called Kakeya maximal function, the restriction properties of the Fourier transform, or the covering properties satisfied by parallelepipeds in $\R^n$ having arbitrary directions and eccentricities. 

There is also the hope that obtaining deep understanding of the Bochner-Riesz operators could be a first step in the project of extending the classical Calder\'{o}n-Zygmund theory of singular integrals, or pseudodifferential operators, going beyond kernels whose singularities are located only at the origin or at infinity, as is demanded in several areas of number theory or PDEs.

\

In the extreme case, $\alpha=0$, the multiplier $T=T_0$ is given by the indicator function of the unit ball. By a remarkable result of C.~Fefferman \cite{5} we know that it is bounded only in the obvious case $p=2$, disproving the conjecture about the boundedness of $T$ in the range $2n/(n+1)<p<2n/(n-1)$.

In its proof Fefferman made use of the properties of the Kakeya sets in the plane (for every $N\gg 1$ there is a set whose measure is less than $1/\log N$ but containing a rectangle of dimensions $1\times 1/N$ on every direction), but also of  a previous result due to Y.~Meyer, who observed that the $L^p$-boundedness of $T$ implies a vector-valued
 control for Hilbert transforms in different directions of the space. More concretely:

Let
\begin{equation}\label{dht} H_{\omega} f(x)=\text{p.v.}
\int_{-\infty}^\infty \frac{f(x-t\omega)}{t}\; dt,
\qquad \omega\in S^{n-1},
\end{equation}
Then $T$ bounded on $L^p(\R^n)$ implies
\[
 \Big\|
\big(\sum|H_{\omega_j}f_j|^2\big)^{1/2}
 \Big\|_p
\lesssim
 \Big\|
\big(\sum|f_j|^2\big)^{1/2}
 \Big\|_p.
\]
(Throughout this paper the symbol $f\lesssim g$ means that there exists $C$ independent of $f$ and $g$ such that $|f(x)|\le C|g(x)|$ for all $x$. We will also write $f\lesssim_s g$ to denote that the involved constant depends upon the parameter $s$).

It was precisely that vector-valued
 inequality, the one disproved by Fefferman with the help of an appropriated Kakeya set, which one may now describe as the main enemy of the  Bochner-Riesz conjecture.

One way of defeating the enemy was introduced in \cite{6} (see also \cite{7} for an independent proof):
\[
T \text{ is bounded on }L^p_{\text{\rm rad}}L^2_{\text{\rm ang}}(\R^n) \text{ if and only if } 
\frac{2n}{n+1}<p<\frac{2n}{n-1},
\]
where the norm in $L^p_{\text{\rm rad}}L^2_{\text{\rm ang}}(\R^n)$ is given by the integral
\[
 \bigg(\int_0^\infty\Big(\int_{S^{n-1}}|f(r\theta)|^2\; d\sigma(\theta)
\Big)^{p/2}r^{n-1}\; dr\bigg)^{1/p}
\]
with  $d\sigma$ the uniform measure in the unit sphere and $r>0$, $\theta\in S^{n-1}$ are the polar coordinates in $\R^n$.

\

Throughout this paper the proof given in \cite{6} will be revisited, improving some of the arguments and estimates. In particular we will show that Meyer's lemma holds in the expected range $2n/(n+1)<p<2n/(n-1)$ so long as we substitute $L^p (\R^n)$ by
$L^p_{\text{\rm rad}}L^2_{\text{\rm ang}}(\R^n)$.

For every $1<p<\infty$, singular integrals (i.e. integral operators whose kernels are given by $\text{p.v.}\|x\|^{-n}\Omega(x/\|x\|)$, with $\int_{S^{n-1}}\Omega=0$) yield   bounded operators on $L^p_{\text{\rm rad}}L^2_{\text{\rm ang}}(\R^n)$. However, it is somehow surprising to realize that directional Hilbert transforms $H_\omega$ defined by \eqref{dht}
are bounded on $L^p_{\text{\rm rad}}L^2_{\text{\rm ang}}(\R^n)$ if and only if $2n/(n+1)<p<2n/(n-1)$.

To see that $p$ must necessarily lie in this range for the boundedness of $H_\omega$ it suffices to check the action of $H_\omega$ over the indicator function of a unit cube. However, the other implication is more involved and we will present two different proofs.

One of them is on the spirit of Meyer's lemma and makes use of the main result in \cite{6}. The second is based on properties of the \emph{universal Kakeya maximal function}:
\begin{equation}\label{ukmf}
 \mathcal Uf(x)=
\sup_{\substack{a,b>0\\ \omega\in S^{n-1}}}
\frac{1}{a+b}\int_{-a}^b \big|f(x+t\omega)\big|\; dt.
\end{equation}
With the help of Kakeya sets it is easy to see that $\mathcal U$ cannot be bounded on any space $L^p(\R^n)$, $1\le p<\infty$. However we will present a geometric argument to show that, acting on radial functions $f$, we have the estimate
\[
 \|\mathcal U f\|_p \lesssim_p 
\|f\|_p
\qquad\text{for $p>n$}.
\]
This was first proved in \cite{10}. Here we will present a new proof using the old \lq\lq bush" and \lq\lq brush" methods of \cite{11}, \cite{12} and \cite{23}. 
\

Throughout this paper several weighted estimates for singular integrals play a crucial role. In particular the fact, discovered in \cite{13}, that for any $s>1$, and any locally integrable function $f$, we have that
$\omega=\big( M(f^s)(x)\big)^{1/s}$ is a weight in the class $A_1$ such that
\[
 M\omega(x) \lesssim_s \omega(x)
\]
where $M$ denotes the Hardy-Littlewood maximal function.

However, this should not be a big surprise, because the main reason to ask the question raised in \cite{13} (circa 1973) was to produce a functional machine to reduce the Bochner-Riesz multiplier problem to the boundedness properties of the Kakeya maximal function, that is
 to the covering properties of parallelepipeds in $\R^n$ (see \cite{14}, \cite{15}). In fact the formula $\big(M(f^s)\big)^{1/s}$ yields perhaps the most interesting examples of $A_p$ weights. The other known cases are the radial powers $|x|^\alpha$, and in this paper we will make use of both classes. In particular we use that $|x|^\alpha\in A_p(\R)$ if and only if $-1<\alpha<p-1$. But we also present an argument showing how the theory for the weights $|x|^\alpha$ can also be deduced from the estimate

\[
M\Big(\big( M(f^s)\big)^{1/s}\Big) \lesssim_s \big( M(f^s)\big)^{1/s}.
\]

\noindent{\bf Acknowledgement.} I am grateful to Diego C\'ordoba, Fernando Chamizo and Keith Rogers for their useful comments and help in the preparation of the manuscript. 

\section{Singular integrals in $L^p_{\text{\rm rad}}L^2_{\text{\rm ang}}
(\R^n)$}

Consider the singular integral operator
\[
 Tf(x)=\text{p.v.}\int_{\R^n}
f(x-y)\frac{\Omega(y)}{|y|^n}dy
\]
where $\Omega$ is smooth enough and has zero mean value
$\int_{S^{n-1}}\Omega(y)d\sigma(y)=0$.

\begin{theorem}\label{singular}
$T$ extends to a bounded operator from $L^p_{\text{\rm rad}}L^2_{\text{\rm ang}}(\R^n)$ to itself when $1<p<\infty$.
\end{theorem}

\begin{proof}
It suffices to prove the theorem when $2\le p<\infty$ because then the case $1<p\le 2$ will follow by duality.

Given $f\in L^p_{\text{\rm rad}}L^2_{\text{\rm ang}}
(\R^n)$ and $g\in\mathcal S(\R)$, let us consider
\[
 I(f,g)=
\int_0^\infty \int_{S^{n-1}}
\big|Tf(r,\theta)\big|^2g(r)r^{n-1}\; d\sigma(\theta)dr
=
\int_{\R^n}
\big|Tf(x)\big|^2g\big(|x|\big)\; dx.
\]
By \cite{13}, we have
\[
 I(f,g)\lesssim_s
\int_{\R^n}
\big|f(x)\big|^2\big(Mg^s(x)\big)^{1/s}\; dx
\]
for $1<s<\infty$, where $Mg^s$ denotes the Hardy-Littlewood maximal operator applied to the function $g^s(x)=\big(g(|x|)\big)^s$. But
\[
 Mg^s(x)
=
\sup_{x\in B}
\frac{1}{|B|}
\int_B g^s(y)\; dy
\]
is also radial. Therefore
\begin{eqnarray*}
 I(f,g)&\lesssim_s&
\int_0^\infty \int_{S^{n-1}}
\big|f(r,\theta)\big|^2 \big(Mg^s(r)\big)^{1/s} r^{n-1}\; d\sigma(\theta)dr
\\
&\lesssim_s&
\|f\|_{L^p_{\text{\rm rad}}L^2_{\text{\rm ang}}
(\R^n)}
\Big(
\int_0^\infty 
\big(Mg^s(r)\big)^{q/s}r^{n-1}\; dr
\Big)^{1/q}
\end{eqnarray*}
so long as $1<s<q$, where $2/p+1/q=1$.

Taking the supremum over all such $g$ so that $\int_0^\infty g^q(r)r^{n-1}\; dr\le 1$, we obtain the proof of Theorem~\ref{singular}.
\end{proof}

The theorem above is somehow the result that one may expect, but the following is, perhaps, more surprising. Let us consider the Hilbert transform $H_\omega$ in the direction of $\omega\in S^{n-1}$, defined in \eqref{dht}. 

\begin{theorem}\label{hbounded}
$H_\omega$ is bounded on 
$L^p_{\text{\rm rad}}L^2_{\text{\rm ang}}
(\R^n)$
if and only if 
$2n/(n+1)<p<2n/(n-1)$.
\end{theorem}

\begin{proof}
Because the rotational symmetry, it suffices to show the conclusion for $\omega=e_n=(0,0,\dots, 1)$. 

To see the \lq\lq only if" part we take $f$ to be the indicator function of the unit cube and observe that
\[
 \big|H_{e_n} f(x_1,x_2,\dots, x_{n-1},y)\big|\gtrsim\frac{1}{|y|}
\]
for $|x_j|\le 1/2$, $j=1,2,\dots, n-1$, $|y|\ge 2$.

Then an elementary computation yields the inequality
\begin{eqnarray*}
\int_0^\infty \int_{S^{n-1}}
\Big(
\big|H_{e_n}f(r,\theta)\big|^2\; d\sigma(\theta)
\Big)^{p/2}  
r^{n-1}\;dr
&\gtrsim&
\int_2^\infty
\Big(\big(\frac{1}{r}\big)^2\frac{1}{r^{n-1}}\Big)^{p/2}
r^{n-1}\;dr
\\
&=&
\int_2^\infty
r^{n-p(n+1)/2}\frac{dr}{r}.
\end{eqnarray*}
therefore the boundedness of $H_{e_n}$ implies that $p>2n/(n+1)$. Again the other bound $p<2n/(n-1)$ follows by duality.

\

To prove the \lq\lq if'' part let us observe first that the 
$L^p_{\text{\rm rad}}L^2_{\text{\rm ang}}
(\R^n)$
norm is well-behaved under dilations and it is preserved under multiplication by characters $e^{ix\cdot \eta}$:

If $f_\delta(x)=f(\delta x)$ then 
\[
 \|f_\delta\|_{ L^p_{\text{\rm rad}}L^2_{\text{\rm ang}}
(\R^n) }
= \delta^{-n/p}\|f\|_{ L^p_{\text{\rm rad}}L^2_{\text{\rm ang}}
(\R^n) }.
\]
Therefore if $\widehat{T_Rf}(\xi)=\chi_{B_R}(\xi)\widehat{f}(\xi)$, where $B_R$ is the ball of radius $R$, we have
\[
 T_Rf(x)=T_1f_{1/R}(Rx)
\]
which produces the estimate
\begin{eqnarray*}
\|T_Rf\|_{L^p_{\text{\rm rad}}L^2_{\text{\rm ang}}
(\R^n)}
=R^{-n/p}\|T_1f_{1/R}\|_{L^p_{\text{\rm rad}}L^2_{\text{\rm ang}}
(\R^n)}                            
&\lesssim&
R^{-n/p}\|f_{1/R}\|_{L^p_{\text{\rm rad}}L^2_{\text{\rm ang}}
(\R^n)}                            
\\
&\lesssim&
\|f\|_{L^p_{\text{\rm rad}}L^2_{\text{\rm ang}}
(\R^n)}                            
\end{eqnarray*}
uniformly on $R>0$, so long as $2n/(n+1)<p<2n/(n-1)$. 

Next we consider the multiplier, $i\;\text{sign}(\xi\cdot \omega)$, corresponding to the Hilbert transform $H_\omega$ and we observe that its boundedness properties are equivalent to those of the Fourier multiplier operator associates to the indicator function of hyperplanes $\{\xi\;:\; \xi\cdot\omega\le t\}$.

Given a function $f$ (chosen so that its Fourier transform $\widehat{f}$ has compact support) the expression
\[
 e^{-2\pi ix\cdot \xi_k}
T_{R_k}\big(
 e^{2\pi iy\cdot \xi_k}f(y)\big)(x)
\]
corresponds to the Fourier multiplier given by the indicator function of the ball of radius $R_k$ centered at $\xi_k$. Choosing a convenient sequence of points $\xi_k$ and radius $R_k$ we can write
\[
 \chi_{\{\xi\cdot\omega\le 0\}}(\xi)\widehat{f}(\xi)
=
\lim_{k\to\infty}
 \chi_{\|\xi-\xi_k\|\le R_k}(\xi)\widehat{f}(\xi).
\]
Then dominated convergence, together with the uniform bounds of $T_{R_k}$ in 
$L^p_{\text{\rm rad}}L^2_{\text{\rm ang}}
(\R^n)$,
$2n/(n+1)<p<2n/(n-1)$, allows to finish the proof. 
\end{proof}

\noindent A different proof of Theorem~\ref{hbounded} is the following:

We start with the integral 
\[
 I(f,g)=
\int_0^\infty \int_{S^{n-1}}
\big|H_\omega f(r,\theta)\big|^2g(r)r^{n-1}\; d\sigma(\theta)dr
\]
We have
\[
 I(f,g)
= 
\int_{\R^n}
\big|H_\omega f(x)\big|^2g\big(|x|\big)\; dx
\lesssim_s
\int_{\R^n}
\big|f(x)\big|^2\big(M_\omega g^s(x)\big)^{1/s}\; dx
\]
where $M_\omega$ denotes the one-dimensional Hardy-Littlewood maximal function in the direction of $\omega$ and $s>1$. 

As $g^s(x)=g^s(|x|)$ is radial one may consider its universal Kakeya maximal function, defined in \eqref{ukmf}, to get the upper bound
\[
 I(f,g)
\lesssim_s
\int_{\R^n}|f(x)|^2 
\big(\mathcal U g^s(x)\big)^{1/s}\; dx
\]
that is
\[
\lesssim_s
\bigg(\int_0^\infty \int_{S^{n-1}}
\Big(
\big|f(r,\theta)\big|^2\; d\sigma(\theta)
\Big)^{p/2}  
r^{n-1}\;dr
\bigg)^{2/p}
\Big(\int_{\R^{n}}
\big( \mathcal U g^s(x)\big)^{q/s}\; dx
\Big)^{1/q}
\]
where $2/p+1/q=1$. To finish we make use of the fact that $\mathcal U$ is bounded on $L^{p}(\R^n)$, $p>n$, when restricted to radial functions \cite{10}:

Observe that $2n/(n+1)<p<2n/(n-1)$ yields 
\[
\frac 1q=1-\frac 2p<1-\frac{n-1}n=\frac 1n.
\]
Therefore given $q>n$ we choose $s>1$ so that $q/s>n$ to obtain
\[
\int_{\R^{n}}
\big( \mathcal U g^s(x)\big)^{q/s}\; dx
\lesssim_s
\int_{\R^n}|g(x)|^q\; dx=
\int_{0}^\infty |g(r)|^q
r^{n-1}\; dr
\]
which coincides with $\|f\|_{  L^p_{\text{\rm rad}}L^2_{\text{\rm ang}}
(\R^n) }$.

\begin{corollary}[Meyer's lemma]\label{meyer}
Given $\{\theta_j\}$, a countable family of directions in $\R^n$, and the corresponding directional Hilbert transforms,
$
 H_jf=H_{\theta_j}f
$, the following inequality holds:
\[
 \Big\|
\big(\sum |H_jf_j|^2\big)^{1/2}
\Big\|_{  L^p_{\text{\rm rad}}L^2_{\text{\rm ang}}
(\R^n) }
\lesssim
 \Big\|
\big(\sum |f_j|^2\big)^{1/2}
\Big\|_{  L^p_{\text{\rm rad}}L^2_{\text{\rm ang}}
(\R^n) }
\]
where $2n/(n+1)<p<2n/(n-1)$.
\end{corollary}

\begin{proof}
Once more it is enough to prove it when $2\le p<2n/(n-1)$ and then use duality to cover the other cases.

Given $g\in\mathcal S(\R)$ we consider the integrals
\[
\sum_j
\int_0^\infty \int_{S^{n-1}}
\big|H_j f_j(r,\theta)\big|^2g(r)r^{n-1}\; d\sigma(\theta)dr
=
\sum_j\int_{\R^n}
\big|H_j f_j(x)\big|^2g\big(|x|\big)\; dx
\]
that is 
\[
\lesssim_s
\sum_j
\int_{\R^n}|f_j(x)|^2 
\big(\mathcal U g^s(x)\big)^{1/s}\; dx
\]
and
\[
\lesssim_s
\bigg|\int_0^\infty \int_{S^{n-1}}
\Big(
\big|\sum_jf_j(r,\theta)\big|^2\; d\sigma(\theta)
\Big)^{p/2}  
r^{n-1}\;dr
\bigg|^{2/p}
\Big|\int_{\R^{n}}
\big( \mathcal U g^s(x)\big)^{q/s}\; dx
\Big|^{1/q}
\]
and the proof follows for the same reasons given in Theorem~\ref{hbounded}.
\end{proof}

\

We finnish this section by presenting a new  proof, using old techniques,  of the boundedness of the general Kakeya maximal function acting on radial functions. 

Related with these problems are maximal functions associated to vector fields in $\R^n$. Given a continuous field of directions, $v(x)$, and a positive valued real function $p(x)$, let us consider
\[
 Mf(x)=\sup_{0<r,t<p(x)}
\frac{1}{r+t}
\int_{-r}^t\big|f(x+sv(x))\big|\; ds.
\]
As it is usual in differentiation theory, the behavior of those maximal operators produces quantitative versions of the Lebesgue's differentation theorem. In general the problem is rather difficult but E.~Stein and S.~Wainger \cite{21}, and also several other authors \cite{9}, have created a theory which yields sufficient conditions for boundedness.

However, for our discussion the following two examples are of special relevance:
\begin{enumerate}
 \item[\bf 1)] 
$v(x)=x/\|x\|$, $p(x)=\infty$. In this case $M$ is bounded on $L^p(\R^n)$ if and only if $p> n$.
\begin{enumerate}
 \item[\bf 1*)] 
$v(x)=x/\|x\|$, $p(x)=\frac 12\|x\|$. Then $M$ is of weak-type (1,1).
\end{enumerate}
 \item[\bf 2)] 
In $\R^2$, 
$v(x,y)=(-y,x)/\|(x,y)\|$, $p(x)=\infty$. Then $M$ is bounded on $L^p(\R^2)$, $1<p<\infty$.
\end{enumerate}
The proof of {\bf 1)} uses polar coordinates together with the fact that $|x|^{n-1}$ is an $A_p$ weight in $\R^n$ for $p>n$:
\begin{eqnarray*}
\int_{\R^n}
\big|M f(x)\big|^p\; dx
&=&
 \int_{S^{n-1}}\int_0^\infty
\big|M f(r,\theta)\big|^pr^{n-1}\; drd\sigma(\theta)
\\
&\lesssim_p&
 \int_{S^{n-1}}\int_0^\infty
\big|f(r,\theta)\big|^pr^{n-1}\; drd\sigma(\theta)
\end{eqnarray*}
so long as $p>n$. The proof of 
{\bf 1*)}
is left as an exercise to the reader. 

Regarding
{\bf 2)}
let us consider in $\R^2$ identified with $\C$, the change of variables $w=\log z$ mapping the straight line $\rho e^{i\theta}(1+ir)$, $r\in\R$, into the curve $\Gamma:\log\rho+i\theta+\log(1+ir)$. But $\Gamma$ is the result of translating the fixed curve $w=\log(1+ir)$ to the point $\log\rho+i\theta$. Therefore, in the $w$-plane, the maximal function is realized as
\[
 \widehat{M}g(w)=
\sup_{0<t,r}
\frac{1}{r+t}
\int_{-r}^t\big|g(w+\gamma(s))\big|\; ds,
\]
where $\gamma(s)=\log(1+is)$. In this setting we can apply the results of E.~Stein and S.~Wainger \cite{21} to finish the proof.

\

The result for the vector field {\bf 2)} can be extended to higher dimensions if we restrict our attention to radial functions, because, in that case, we can take any direction perpendicular to the position vector $x$ without changing the value of the maximal function there.

That is, considering the product structure $\R^n=\R\times \R^{n-1}$ the evaluation of the maximal function at the point $(x_1,x_2,\dots, x_n)$ can be obtained as the evaluation of its $n-1$-dimensional version acting on the slice $\{x_1\}\times \R^{n-1}$. Then an induction argument together with Fubini's theorem gives the boundedness for every $p>1$.

Our proof for $\mathcal U$ will be then modeled upon these two extreme examples {\bf 1)} and {\bf 2)}, and it will make use of the geometry of parallelepipeds in $\R^n$ to treat the intermediate case: 

In $\R^n$, $n\ge 2$, the universal Kakeya maximal function \eqref{ukmf} is unbounded on every $L^p(\R^n)$, $1\le p<\infty$, as an appropriated Kakeya set argument easily shows. However, acting on radial functions $f$ we have

\begin{proposition}\label{urad}
$\|\mathcal U f\|_p\lesssim \|f\|_p$ if $p>n$ and $f$ is radial. 
\end{proposition}

\begin{proof}
Fixing a radial function $f$ and a positive number $\lambda>0$ let us consider the set
\[
 E_\lambda=\big\{ x\;:\; \mathcal U f(x)>\lambda\big\}.
\]
Given $x\in E_{\lambda}$ there is a set of directions $\Gamma_x\subset S^{n-1}$ so that 
\[
\sup_{\substack{a,b>0\\ \omega\in S^{n-1}}}
\frac{1}{a+b}\int_{-a}^b \big|f(x+t\omega)\big|\; dt
\ge \lambda
\]
for every $\omega\in\Gamma_x$. Clearly if $\|x\|=\|y\|$ then $\Gamma_x$ and $\Gamma_y$ coincide after a rotation. Furthermore, for a given $x$ and $\omega\in \Gamma_x$ we will choose $a,b$ satisfying the inequality above and such that its sum $a+b$ is the biggest possible; then among all such $\omega\in S^{n-1}$, $a$ and $b$, we select those maximizing the projection of the segment 
$\big\{x+s\omega\;:\; -a\le s\le b\big\}$
into the radial direction.

Taking advantage of the radial symmetry and after some elementary geometrical considerations, the set $E_\lambda$ is contained in an \lq\lq almost disjoint'' union of annuli:
\[
 E_\lambda=
\bigcup_j
\big\{R_j^1\le |x|\le R_j^2\big\}=\bigcup C_j
\]
satisfying the following properties:
\begin{enumerate}
 \item 
For each annulus $C_j$ there exists a direction $\omega_j$ so that the average of $|f|$ on the segments of directions $\omega_j$, starting at some point of the sphere $\|x\|=R_j^1$ and ending on the sphere  $\|y\|=R_j^2$, is the given value $\lambda$.
 \item 
We have that $C_j\cap C_k=\emptyset$ if $|k-j|\ge 2$.
\end{enumerate}
Clearly, in order to estimate the size of $E_\lambda=\{\mathcal U f>\lambda\}$ it is enough to control, independently, the portion of that set contained in each annulus $C_j$.

We shall distinguish two classes of annuli:
\begin{enumerate}
 \item[i)] 
Thin shells: $\big\{ R\le |x|\le R+\Delta;\ \Delta<\frac 12 R\big\}$.
 \item[ii)] 
Thick shells: $\big\{ R\le |x|\le R+\Delta;\ \Delta\ge\frac 12 R\big\}$.
\end{enumerate}
In both cases (thin or thick) we fix two poles $N$ and $S$ (North and South) in order to estimate the portion of $E_\lambda$ near the equator. Given the radial symmetry we can select the direction $\omega$ in such a way that the ray from each point of the shell meets the $NS$  axis, but we can also establish the convention that, among the two possible rays, the chosen one is pointing north.

For a thick annulus $C$ its volume is comparable to $(\Delta+R)^n$. Therefore one needs to show the inequality
\[
 (\Delta+R)^n\lesssim_p
\frac{1}{\lambda^p}
\int_C |f(x)|^p\; dx
\qquad
\text{for any }p>n.
\]
Let us observe that for each point in the sphere $\|x\|=R+\Delta$ we are then given a straight line segment $L_x$ starting at $x$; ending at the inner sphere $\|y\|=R$; tangent to a certain sphere $\|y\|=R_0\le R$ and such that
\begin{enumerate}
 \item 
The average of $|f|$ on $L_x$ is $\lambda$.
\item
The straight line $L_x$ intersects the $NS$ axis.
\end{enumerate}
Again, given the fixed radial function $f$ one can enlarge the segment $L_x$ (with some small positive number $\epsilon$) to become a tube of radius $\epsilon$. In such a way that the average of $|f|$ in that tube is bigger than, says, $\lambda/2$.

Next we cover $E_\lambda$ (or a fixed portion of it close to the equator) with a family of those tubes but keeping them pairwise disjoint at the outer sphere $\|x\|=R+\Delta$.

That is, we have obtained what is called a \emph{brush configuration} of tubes, meeting the $NS$ axis and being disjoint at the outer sphere of the shell.

Let us denote by $\big\{T_\nu^\epsilon\big\}_\nu$ the collection of those tubes and let us consider their overlapping function 
$\sum_\nu \chi_{ T_\nu^\epsilon }(x)$.
Then some elementary geometric considerations (\lq\lq brush argument'') yields the following estimate:
\begin{equation}\label{mover}
 \big|
\big\{
\sum_\nu \chi_{ T_\nu^\epsilon }(x)\ge d>0
\big\}
\big|
\lesssim
\frac{(\Delta+R)^n}{d^{n/(n-1)}}.
\end{equation}
Therefore
\begin{eqnarray*}
 |C|
&\lesssim&
\sum_\nu\big| T_\nu^\epsilon \big|
\lesssim
\frac{1}{\lambda}
\sum_\nu
\int_{ T_\nu^\epsilon }|f(x)|\; dx
\le
\frac{1}{\lambda}
\sum_\nu
\int|f(x)|\sum_\nu \chi_{ T_\nu^\epsilon }(x)\; dx
\\
&\le &
\frac{1}{\lambda}
\|f\|_p\big\|\sum_\nu \chi_{ T_\nu^\epsilon }
\big\|_q
\lesssim
\frac{1}{\lambda}
\|f\|_p
|C|^{1/q}
\qquad\text{with }1/p+1/q=1,
\end{eqnarray*}
so long as $q< n/(n-1)$ (equivalently $p>n$) allowing us to finish the proof. 

\

The elementary geometrical considerations are the following:

\

1) For each $\theta$ in the equator  $S^{n-2}$ of $S^{n-1}$, let us consider the two dimensional plane $H_\theta$ determined by $\theta$ and the $NS$ axis, and also the enlarged band $H_\theta^\epsilon=\big\{x\;:\; \text{dist}(x,H_\theta)\le \epsilon/2\big\}$. Choosing conveniently $\epsilon$-spaced points $\{\theta_j\}$ in $S^{n-2}$, we may assume that the sets $H_{\theta_j}^\epsilon\cap \{|x|=R+\Delta\}$ are pairwise disjoint and their union cover $\{|x|=R+\Delta\}$. Then their overlapping where $|x|<R+\Delta$ is easily controlled: at distance $R_0<r<R+\Delta$ we have
\[
 \sum \chi_{ H_{\theta_j}^\epsilon }(x)\lesssim
\Big(\frac{\Delta+R}{r}\Big)^{n-2}.
\]
Then the collection of tubes $\big\{T_\nu^\epsilon\big\}$ is divided into disjoint classes by the inclusion relation $T_\nu^\epsilon\subset H_{\theta_j}^\epsilon $.

\

2) The tubes inside $H_{\theta_j}^\epsilon $ meets at the tangential inner sphere $\|x\|=R_0\le R$ producing a collection of  bush configurations (that is sets of tubes meeting at a common point in $R_0$ ) such that each tube inside $H^\epsilon_{\theta_j}$ belong to, at most, two of those bushes.  Then their overlapping function is easily controlled by an elementary calculation (bush argument):
\[
 \sum_{T_\nu^\epsilon \subset H_{\theta_j}^\epsilon } 
\chi_{ T_\nu^\epsilon }(x)
\lesssim
\frac{\Delta+R}{r}
\]
when $R_0<|x|=r<R+\Delta$.

\

Those two estimates together imply that in the sphere $|x|=r$, $R_0<r<R+\Delta$ the overlapping is bounded by $(R+\Delta)^{n-1}/r^{n-1}$ uniformly in $\epsilon>0$. That is, we get \eqref{mover}.

\

To treat thin shells we proceed in the same manner, but here the estimate is reduced easily to each plane $H_\theta$. That is, we only have to consider the two dimensional case and observe that the $\epsilon$-rectangle $R_{\theta, j}^\epsilon$ arranged in disjoint families of bushes satisfying the overlapping estimate
\[
  \big|
\big\{
\sum \chi_{ R_{\theta, j}^\epsilon }(x)\ge d
\big\}
\big|
\lesssim
\frac{\Delta R}{d^2}
\]
which implies
\[
 \int_{S^{n-2}}
 \big|
\big\{
\sum \chi_{ R_{\theta, j}^\epsilon }(x)\ge d
\big\}
\big|
R^{n-2}\;d\theta
\lesssim
\frac{\Delta R^{n-1}}{d^2}.
\]
\end{proof}

\section{The disc multiplier revisited}

With the same notation used in the introduction, let us consider
\[
 \widehat{Tf}(\xi)=\chi_B(\xi)\widehat{f}(\xi)
\]
for rapidly decreasing smooth functions $f$, where $B$ is the unit ball in $\R^n$, $n\ge 2$.

\begin{theorem}\label{disc}
The operator $T$ is bounded on  
$L^p_{\text{\rm rad}}L^2_{\text{\rm ang}}(\R^n)$
if and only if  
${2n}/{(n+1)}<p<{2n}/{(n-1)}$,
\end{theorem}

The proof of this theorem was first given in \cite{6} (see also \cite{7} for an independent proof and \cite{17} for a weighted version). Here we will improve and simplify our previous presentation in order to motivate subsequent results.

\

The \lq\lq only if'' part follows easily taking $f$ to be the inverse Fourier transform of a $C^\infty$-function, $\widehat{f}$, such that 
$\widehat{f}\equiv 1$ when  $|\xi|\le 1$
and
$\widehat{f}\equiv 0$ if  $|\xi|\ge 2$.
Then 
\[
 Tf(x)=C_n|x|^{-n}J_{n/2}(2\pi|x|)
\]
which belongs to 
$L^p_{\text{\rm rad}}L^2_{\text{\rm ang}}(\R^n)$
if and only if $p>2n/(n+1)$. The other bound $p<2n/(n-1)$ follows by duality.

Given $f\in\mathcal S(\R^n)$ (the Schwartz class of rapidly decreasing smooth functions) it has a development
\[
 f(x)=\sum_{k,\ell} f^{\ell}_k(|x|)Y^{\ell}_k\big(\frac{x}{|x|}\big),
\]
where the $f^\ell_k$ are defined on $[0,\infty)$ and $\{Y^\ell_k\}$ is an orthonormal basis of the $d_k$-dimensional space of spherical harmonic polynomials of degree $k$ in $\R^n$. Here, the index $k$ takes nonnegative integer values and, for each $k$, $\ell$ takes values in the interval $1\le \ell\le d_k={n+k-1 \choose k}-{n+k-3 \choose k-2}$.

It is a well-known
 fact (see \cite{22}) that the Fourier transform preserves that development:
\[
\widehat{f}(\xi)=\sum_{k,\ell} i^kF^\ell_k(|\xi|)Y^\ell_k\big(\frac{\xi}{|\xi|}\big),
\]
where
\[
 F^\ell_k(r)=
r^{-n/2+1}
\int_0^\infty f^\ell_k(s) s^{n/2}J_{k-1+n/2}(rs)\;ds.
\]
Here $J_\nu$ denotes the Bessel function of order $\nu\geq 0$, which evaluated at a non-negative real number $x$ is given  by the integral
\[
 J_\nu( x)= \frac{1}{\pi}
\int_0^{\pi}\cos(\nu t-x\sin t)\; dt - \frac{sin\pi\nu}{\pi}\int_0^{\infty} e^{-\nu t- xsinh t}\; dt.
\]
(\cite{16}, page 176, formula 4).

Continuing with the proof of the theorem, we note that 
\[
 Tf(x)=\sum_{k=0}^\infty
T_k^n f^\ell_k(|x|)Y^\ell_k\big(\frac{x}{|x|}\big)
\]
where $T_k^n$ are integral operators given by the formula
\[
 T_k^ng(t)= (-1)^kt^{-(n-1)/2}
\int_0^\infty f(r) r^{(n-1)/2}K_{k-1+n/2}(t,r)\;dr.
\]
and 
\[
 K_{\nu}(t,r)
=
\sqrt{tr}
\int_0^1 J_\nu(rs)J_\nu(ts) s\; ds.
\]

\

Let us begin decoding the kernels (see also ref. \cite{16}). 
Consider the ODE verified by the Bessel functions
\[
 x^2J_\nu''(x)+xJ_\nu'(x)+(x^2-\nu^2)J_\nu(x)=0
\]
which written in terms of the functions
\[
 \mathcal U_r(s)=\sqrt{rs}J_\nu(rs)
\qquad\text{and}\qquad
q_r(s)=\frac 14 s^{-2} +r^2-\nu^2s^2
\]
becomes
\[
 \mathcal U''_r(s)+q_r(s)\mathcal U_r(s)=0.
\]
Then
\begin{eqnarray*}
(t^2-r^2)\sqrt{rt}\int_0^1J_\nu(rs)J_\nu(st)s\; ds
&=&
\int_0^1\big(q_r(s)-q_t(s)\big)
\mathcal U_r(s)
\mathcal U_t(s)\; ds
\\
&=&
\int_0^1
\big(
\mathcal U_r(s)
\mathcal U_t''(s)-
\mathcal U_r''(s)
\mathcal U_t(s)
\big)
\; ds
\\
&=&
\mathcal U_r(1)
\mathcal U_t'(1)-
\mathcal U_r'(1)
\mathcal U_t(1).
\end{eqnarray*}
Therefore
\begin{eqnarray*}
 K_\nu(t,r)
&=&
\sqrt{rt}
\frac{tJ_\nu'(t)J_\nu(r)-rJ_\nu'(r)J_\nu(t)}{t^2-r^2}
\\
&=&
\frac{\sqrt{t}J_\nu'(t)J_\nu(r)\sqrt{r}}{2(t-r)}
+
\frac{\sqrt{t}J_\nu'(t)J_\nu(r)\sqrt{r}}{2(t+r)}
\\
&&
+
\frac{\sqrt{t}J_\nu(t)J_\nu'(r)\sqrt{r}}{2(r-t)}
+
\frac{\sqrt{t}J_\nu(t)J_\nu'(r)\sqrt{r}}{2(r+t)}
=\sum_{j=1}^4 K_\nu^j(t,r).
\end{eqnarray*}
Thus we have obtained four families of the integral operators
\[
 K_\nu^j f(t)=\int_0^\infty K_\nu^j(t,r)f(r)\; dr.
\]
However, taking into account the asymmptotics of Bessel functions (see \cite[p.\;199]{16}):
\begin{eqnarray*}
 J_\nu(z)
\sim
\big(\frac{2}{\pi z}\big)^{1/2}
\Big(
\cos\big(z-\frac{2\nu+1}{4}\pi\big)
\sum_{m=0}^\infty
\frac{(-1)^m(\nu,2m)}{(2z)^{2m}}
\qquad \ \qquad
&&
\\
-
\sin\big(z-\frac{2\nu+1}{4}\pi\big)
\sum_{m=0}^\infty
\frac{(-1)^m(\nu,2m+1)}{(2z)^{2m+1}}
\Big),
&&
\end{eqnarray*}
where $(\nu,k) = \frac{\Gamma(\nu+k+\frac12)}{k!\Gamma(\nu-k+\frac12)}$,  
one may infer heuristically that the kernels $j=1$ and $j=3$ correspond to Hilbert transforms, while the cases  $j=2$ and $j=4$, are less singular and produce Hardy integral operators. Since the order $\nu$ of the Bessel functions is increasing, the asymptotic estimate above does not produce uniform bounds, and the more precise analysis of the following lemma is needed in the critical range $\nu/2\le r\le 2\nu$.

\begin{lemma}\label{vdc}
The following estimates hold uniformly on $\nu\ge 1$.
\begin{eqnarray*}
\text{i)}&&
\big|J_\nu(r)\big|\lesssim \frac{1}{r^{1/2}};
\quad
\big|J_\nu'(r)\big|\lesssim \frac{1}{r^{1/2}},
\quad
\text{when }r\ge \frac{3}{2}\nu.
\\
\text{ii)}&&
\big|J_\nu(r)\big|\lesssim \frac{1}{1+\nu};
\quad
\big|J_\nu'(r)\big|\lesssim \frac{1}{(1+\nu)^2},
\quad
\text{when }r\le \frac{3}{2}\nu.
\\
\text{iii)}&&
\big|J_\nu(\nu+\rho\nu^{1/3})\big|\lesssim \frac{1}{\rho^{1/4}\nu^{1/3}}
\quad\text{ and }
\\
&&
\big|J_\nu'(\nu+\rho\nu^{1/3})\big|\lesssim \frac{\rho^{1/4}}{\nu^{2/3}},
\quad
\text{when }1\le\rho< \frac{1}{2}\nu^{2/3}.
\\
\text{iv)}&&
\big|J_\nu(\nu-\rho\nu^{1/3})\big|\lesssim \frac{1}{\rho \nu^{1/3}}
\quad\text{ and }
\\
&&
\big|J_\nu'(\nu-\rho\nu^{1/3})\big|\lesssim \frac{1}{\rho^2\nu^{2/3}},
\quad
\text{when }1\le\rho< \nu^{2/3}.
\end{eqnarray*}
\end{lemma}

\noindent As an indication for the reader we sketch the proof of iii). The stationary phase method, van der Corput estimates, and integration by parts allow us to control the size of $J_\nu( x)$ in the critical interval $\frac{\nu}{2}\leq x\leq 2\nu$:

We have
\[
 J_\nu(x)
=
\Re\frac{1}{\pi}\int_0^{\pi}
e^{i(\nu t-x\sin  t)}\; dt + O(\frac{1}{\nu})
=
\Re\frac{1}{\pi}\int_0^{\pi}
e^{ix(\theta t-\sin  t)}\; dt + O(\frac{1}{\nu})
\]
where $\theta=\nu/x\approx \nu/(\nu+\rho\nu^{1/3})\approx 1-\rho \nu^{-2/3}$.
The critical point corresponds to 
\[
 \frac{d}{dt}(\theta t-\sin t)=0,
\]
i.e. $t_\theta=\frac{1}{\pi}\cos^{-1}\theta$. Then 
\[
 \frac{d^2}{dt^2}(\theta t-\sin t)\Big|_{t=t_\theta}=
\sqrt{1-\theta^2}
\approx \sqrt{2\rho\nu^{-2/3}}
\approx \rho^{1/2}\nu^{-1/3}.
\]
Therefore stationary phase method yields

\[
 \big|J_\nu(x)\big|
\lesssim 
\frac{1}{\sqrt{x}}\cdot\frac{1}{\sqrt{\rho^{1/2}\nu^{-1/3}}}
\lesssim
\frac{1}{\rho^{1/4}\nu^{1/3}}
\]
because $x\approx \nu+\rho\nu^{1/3}$, $\rho<\nu^{2/3}$.

The other cases follow by similar arguments (see ref. \cite{16}).

\begin{corollary}\label{prodJ}
In the critical range $\nu/2\le r\le 2\nu$, uniformly in $\nu$, the following estimate hold:
\begin{eqnarray*}
\frac{1}{\nu}
\int_{\nu/2}^{2\nu}
\big|J_\nu(r)r^{1/2}\big|^p\; dr\le C_p,\text{ for }p<4.
\end{eqnarray*}
\end{corollary}

Nevertheless, we fortunately have the combination $\sqrt{t}J_\nu'(t)J_\nu(r)\sqrt{r},$ allowing us to make use of the extra decay hidden there to achieve the estimates.

\begin{proposition}\label{kj}
Given $4/3<p<4$, there exists a finite constant $C_p$  such that
\[
\int_0^\infty\Big(\sum_\ell \big|K_{\nu(\ell)}^j g_\ell(r)\big|^2\Big)^{p/2}\; dr
\le C_p
\int_0^\infty\Big(\sum_\ell \big|g_\ell(r)\big|^2\Big)^{p/2}\; dr
\]
for $j=1,2,3,4$, any function $\nu:\mathbb{N}\to \mathbb{R}_+$ and every sequence of rapidly decreasing smooth functions~$\{g_\ell\}$.
\end{proposition}

\begin{proof}
We will consider the operator $K_\nu^1$ with kernel 
\[
K_\nu^1(t,r)=
 \frac{\sqrt{t}J_\nu'(t)J_\nu(r)\sqrt{r}}{2(t-r)}.
\]
The estimate for $K_\nu^3$ will follow as a result of duality; the cases $K_\nu^2$ 
and 
$K_\nu^4$ being less singular are easier to handle and the details will be left to the reader.
 
Fixing $\nu>0$, we consider the partition 
\[
 [0,\infty)=
[0,\frac 12 \nu)\cup [\frac 12\nu,2\nu)\cup [2\nu,\infty)
=
I_0^\nu
\cup I_c^\nu
\cup I_\infty^\nu
\]
and the corresponding splitting of the kernel. 
Let us denote by $T_{\alpha,\beta}^\nu$ the integral operator with the kernel
\[
 \chi_{I_\alpha^\nu}(t)K^1_\nu(t,r)\chi_{I_\beta^\nu}(t),
\qquad\alpha,\beta = 0,c,\infty.
\]

We consider four cases.

\

\noindent{\bf Case 1:} $\alpha,\beta=0,\infty$ or $\{\alpha,\beta\}=\{0,c\}$.

The cases $\alpha,\beta=0,\infty$ are the easier, because in $I_0^\nu\cup I_\infty^\nu$ both 
$\sqrt{t}J_\nu'(t)$
and 
$\sqrt{t}J_\nu(t)$
are uniformly bounded, and the estimate in the proposition is reduced to a well-known inequality for the Hilbert transform.

Similarly when $\alpha=0$, $\beta=c$ or $\alpha=c$, $\beta=0$; we have that 
\[
 \sup_{t\in I_0^\nu}
\big|\sqrt{t}J_\nu'(t)\big|
 \sup_{r\in I_c^\nu}
\big|\sqrt{r}J_\nu(r)\big|
+
 \sup_{t\in I_0^\nu}
\big|\sqrt{t}J_\nu(t)\big|
 \sup_{r\in I_c^\nu}
\big|\sqrt{r}J_\nu'(r)\big|
\lesssim 1
\]
uniformly in $\nu$ (Lemma~\ref{vdc}).

\

\noindent{\bf Case 2:} $\alpha=c$, $\beta=\infty$.

Let us consider the partition of the interval $I_c^\nu$ through the sets
\[
 G_\tau^+
=\big[\nu+\tau\nu^{1/3},\nu+(\tau+1)\nu^{1/3}\big)
\text{ and }G_\tau^+
=\big[\nu-(\tau+1)\nu^{1/3},\nu-\tau\nu^{1/3}\big)
\]
for $\nu=0,1,2,\dots,\big[\frac 12 \nu^{2/3}\big]$.

We have
\begin{eqnarray*}
 &&
\sum_\tau
\chi_{G_\tau^+}(t)
\big|
\sqrt{t}J_\nu'(t)
\big|
\Big|
\int_{2\nu}^\infty
\frac{\sqrt{r}J_\nu(r)}{t-r}
f(r)\chi_{I_\infty^\nu}(r)\; dr
\Big|
\\
&&\qquad\ \qquad
+
\sum_\tau
\chi_{G_\tau^-}(t)
\big|
\sqrt{t}J_\nu'(t)
\big|
\Big|
\int_{2\nu}^\infty
\frac{\sqrt{r}J_\nu(r)}{t-r}
f(r)\chi_{I_\infty^\nu}(r)\; dr
\Big|
\\
&&\lesssim
\nu^{-1/6}
\sum_\tau
\big(
\tau^{1/4}\chi_{G_\tau^+}(t)
+
\tau^{-2}\chi_{G_\tau^-}(t)
\big)
\Big|
H\big(
f(r)\sqrt{r}J_\nu(r)\chi_{I_\infty^\nu}(r)
\big)(t)
\Big|
\\
&&\lesssim
\Big|
H\big(
f(r)\sqrt{r}J_\nu(r)\chi_{I_\infty^\nu}(r)
\big)(t)
\Big| = |H(f(r)\theta_\nu (r))(t)|
\end{eqnarray*}
where $\theta_\nu (r)= \big| \sqrt{r}J_\nu(r)\chi_{I_\infty^\nu}(r) \big|\lesssim 1$ uniformly in $\nu$.

\

\noindent{\bf Case 3:} $\alpha=\infty$, $\beta=c$.

Since $\big| \sqrt{r}J_\nu(r) \big|\lesssim 1$ uniformly in $\nu$ when $r\ge \frac 32\nu$, the part of the estimate corresponding to that region trivializes.

For the remainder we have, with the notation of the previous case,
\begin{eqnarray*}
&&\chi_{[2\nu,\infty)}(t)
\Big(
\sum_{\tau}
\frac{1}{|\nu|}
\int_{G_\tau^+}
\frac{|f(r)|}{(1+\tau)^{1/4}\nu^{1/3}}\;dr
+
\sum_{\tau}
\frac{1}{|\nu|}
\int_{G_\tau^-}
\frac{|f(r)|}{(1+\tau)\nu^{1/3}}\;dr
\Big)
\\
&&
\lesssim 
\chi_{[2\nu,\infty)}(t)
Mf(t),
\end{eqnarray*}
where as before $Mf$ denotes the Hardy-Littlewood maximal function. 

\

Putting all the estimates together for the cases 1, 2 and 3, we get
\begin{eqnarray*}
\int_0^\infty
\Big|
\sum
\big| T_{\alpha,\beta}^\nu f_\ell\big|^2
\Big|^{p/2}
\hspace{-3pt}
&\lesssim&
\hspace{-3pt}
\int_0^\infty
\Big|
\sum
\big| H( f_\ell\theta_\nu)\big|^2
\Big|^{p/2}
+
\int_0^\infty
\Big|
\sum
\big| M( f_\ell)\big|^2
\Big|^{p/2}
\\
&\lesssim&
\int_0^\infty
\Big|
\sum
\big| f_\ell\big|^2
\Big|^{p/2}.
\end{eqnarray*}

\

\noindent{\bf Case 4:} $\alpha=c$, $\beta=c$.

It will be convenient to modify the splitting and for a fixed $\nu$ let us define
\[
 I_k^+
=\big[\nu+2^k\nu^{1/3},\nu+2^{k+1}\nu^{1/3}\big)
,\qquad
 I_k^-
=\big[\nu-2^{k+1}\nu^{1/3},\nu-2^{k}\nu^{1/3}\big)
\]
where $0\leq k \leq \big[\frac23 \log_2 \nu\big]$,
and $I_0= 
\big[\nu-\nu^{1/3},\nu+\nu^{1/3}\big)$. Define also
\[
I^+=\bigcup_k  I_k^+
\qquad\text{and}\qquad 
I^-=\big(\bigcup_k  I_k^-\big)\cup I_0.
\]
Then we have the kernels
\[
 \chi_{I^\pm}(t)
\frac{\sqrt{t}J_\nu'(t)J_\nu(r)\sqrt{r}}{t-r}
 \chi_{I^\pm}(t).
\]
Taking into account once more the estimate of Lemma~\ref{vdc}, it is obvious that the most dangerous is the $++$ term. Therefore, in the following, we shall present the details of the proof corresponding to that case, leaving the other three as exercises.

To estimate the action of the operator with kernel 
$
 \chi_{I^+}(t)
K_\nu^1(t,r)
 \chi_{I^+}(t)
$
on a function $f_\ell$ at a point $t\in I_k^+$, we divide the integration in three parts:
\begin{align*}
 A_k (t)&=\chi_{I_k^+}(t)
\int\sum_{j\le k-1}
K_\nu^1(t,r)\chi_{I_j^+}(r)f_\ell(r)\; dr,
\\
 B_k (t)&=\chi_{I_k^+}(t)
\int\sum_{j\ge k+1}
K_\nu^1(t,r)\chi_{I_j^+}(r)f_\ell(r)\; dr,
\\
 C_k (t)&=\chi_{I_k^+}(t)
\int
K_\nu^1(t,r)\chi_{I_k^+}(r)f_\ell(r)\; dr.
\end{align*}

First we estimate $A$. We have
\begin{eqnarray*}
|A_k (t)|
&\lesssim&
\sum_{j\le k-1}
\big|\sqrt{t}J_\nu'(t)\big|
\Big|
\int_{I_j^+}
\frac{\sqrt{r}J_\nu(r)}{t-r}f_\ell(r)\; dr
\Big|
\\
&\lesssim&
\sum_{j\le k-1}
2^{-3k/4}
\int_{I_j^+}
\big|J_\nu(r)f_\ell(r)\big|\; dr
\end{eqnarray*}
that gives
\begin{eqnarray*}
|A_k (t)|
&\lesssim&
2^{-3k/4}
\int_\nu^{\nu+2^k\nu^{1/3}}
\big|f_\ell(r)\big|
\sum_{j\le k-1}
\frac{2^{-j/4}}{\nu^{1/3}}
\chi_{I_j^+}(r)
\; dr
\\
&\lesssim&
2^{-3k/4}\nu^{-1/3}
\Big(
\int_\nu^{\nu+2^k\nu^{1/3}}
\big|f_\ell(r)\big|^s
\; dr
\Big)^{1/s}
\big(
\sum_{j\le k-1}2^{-js'/4}2^j \nu^{1/3}
\big)^{1/s'}
\end{eqnarray*}
where $s$ and $s'$ are conjugate H\"older exponents. Observe that if $s>4/3$ then $s'<4$ and we have
\[
\big(
\sum_{j\le k-1}
2^{j(1-s'/4)}
\big)^{1/s'}
\approx
2^{k(1-s'/4)/s'}
\]
and therefore
\[
|A_k (t)|
\lesssim
\Big(
2^{-k}\nu^{-1/3}
\int_\nu^{\nu+2^k\nu^{1/3}}
\big|f_\ell(r)\big|^s
\; dr
\Big)^{1/s}
\lesssim
\big( Mf_\ell^s(t)\big)^{1/s}
\qquad\text{for }s>\frac 43.
\]

Next we estimate $B$:
\begin{eqnarray*}
|B_k (t)|
&\lesssim&
\sum_{j\ge k+1}
\big|\sqrt{t}J_\nu'(t)\big|
\Big|
\int_{I_j^+}
\frac{\sqrt{r}J_\nu(r)}{t-r}f_\ell(r)\; dr
\Big|
\\
&\lesssim&
\sum_{j\ge k+1}
\nu\frac{2^{k/4}}{\nu^{2/3}}\frac{1}{2^j\nu^{1/3}}
\int_{I_j^+}
\big|f_\ell(r)\big|
\frac{2^{-j/4}}{\nu^{1/3}}\; dr
\end{eqnarray*}
that gives
\[
|B_k (t)|
\lesssim
\sum_{j\ge k+1}
2^{-|j-k|/4}\frac{1}{2^j\nu^{1/3}}
\int_{\nu+2^j\nu^{1/3}}^{\nu+2^{j+1}\nu^{1/3}}
\big|f_\ell(r)\big|
\; dr
\lesssim Mf(t)
\]
where we have applied several times the estimates of Lemma~\ref{vdc}.

Finally one needs to control the diagonal terms $C$:
\[
 C_k (t)
=
\sqrt{t}J_\nu'(t)
H\big(\chi_{I_k^+}(r)f_\ell(r)\sqrt{r}J_\nu(r)\big)(t)\chi_{I_k^+}(t)
\]
We have
\begin{eqnarray*}
(\sum_k\big|C_k (t)\big|^2)^\frac12
&\le &
(\sum_k
|\chi_{I_k^+}(t)\nu^{-1/6}2^{k/4}
\big|
H\big(\chi_{I_k^+}(r)f_\ell(r)\sqrt{r}J_\nu(r)\big)(t)
\big|^2)^\frac12
\\
&\le &
\Big(
\sum_k
\big|
H\big(\chi_{I_k^+}(r)f_\ell(r)\nu^{-1/6}2^{k/4}\sqrt{r}J_\nu(r)\big)(t)
\big|^2
\Big)^{1/2}
\\
&=&
\Big(
\sum_k
\big|
H\big(\chi_{I_k^+}(r)f_\ell(r)\theta^k_\nu(r)\big)(t)
\big|^2
\Big)^{1/2}
\end{eqnarray*}
where $\big|\theta^k_\nu(r)\big|\lesssim 1$ uniformly in $\nu$ and $k$.

That is, for each $\alpha, \beta$, we have obtained the pointwise estimate 
\begin{eqnarray*}
(\sum_\ell |T_{\alpha \beta}f_\ell|^2)^\frac12
&\lesssim& 
(\sum_\ell |H(f_\ell\cdot\theta_\nu)|^2)^\frac12 + (\sum_\ell |Mf^s_\ell|^\frac2s)^\frac12 
\\
&&+\, (\sum_\ell \sum_k |H(f_\ell\cdot\theta^k_\nu\cdot\chi_{I_k^+})|^2)^\frac12,
\end{eqnarray*}
for $s>\frac43$. Where the functions $\theta^\nu$, $\theta^\nu_k$ are uniformly bounded.

Therefore, we can finish the proof of Proposition~\ref{kj} with the help of some well-known estimates for the Hilbert transform and the Hardy-Littlewood maximal functions (see \cite{18}, \cite{19}, \cite{20}), namely
\begin{align*}
&\int\Big(
\sum \big| Hf_\ell(x)\big|^s
\Big)^{p/s}\omega(x)\; dx
+ \int\Big(
\sum \big| Mf_\ell(x)\big|^s
\Big)^{p/s}\omega(x)\; dx
\\
&\qquad\ \qquad\lesssim
C_{p,s}(\omega)
 \int\Big(
\sum \big| f_\ell(x)\big|^s
\Big)^{p/s}\omega(x)\; dx.
\end{align*}
So long as the weight $\omega$ belongs to the class $A_p$, that is
\[
 \sup_I
\big(\frac{1}{|I|}
\int_I\omega
\big)
\big(\frac{1}{|I|}
\int_I\omega^{-1/(p-1)}
\big)^{p-1}
=|\omega|_p<\infty.
\]
Furthermore the constant $C_{p,s}(\omega)$ above depends only upon $p$, $s$ and $|\omega|_p$.
\end{proof}

\begin{proof}[Proof of Theorem~\ref{disc}.]
Let us note that in order to prove the theorem it is enough to show the following inequality:
\begin{equation}\label{ineqth4}
\int_0^\infty
\Big|
\sum_\ell \big|K_\nu g_\ell(r)\big|^2
\Big|^{p/2}
r^{(n-1)(1-p/2)}\; dr
\lesssim
\int_0^\infty
\Big|
\sum_\ell \big|g_\ell(r)\big|^2
\Big|^{p/2}
r^{(n-1)(1-p/2)}\; dr
\end{equation}
so long as $2n/(n+1)<p<2n/(n-1)$.

But this will be a consequence of Proposition~\ref{kj} together with some observations about weights in the class $A_p$.

Since it happens that $|x|^\beta\in A_p$ if and only if $-1<\beta<p-1$, we can conclude that $|x|^{(n-1)(1-p/2)}\in A_p$ if and only if $2n/(n+1)<p<2n/(n-1)$.
Furthermore, if $2\le p< 2n/(n-1)$ we have that $|x|^{(n-1)(1-p/2)}\in A_q$, for every $q\ge 1$.  However in the interval $2n/(n+1)<p\le 2$ we have that
 $|x|^{-p/s'+(n-1)(1-p/2)}\in A_{p/s}$
under the hypothesis $s<4$, $1/s+1/s'=1$.

The part of the equality \eqref{ineqth4} which corresponds to the critical intervals:
\[
K_\nu^c(t,r)= \chi_{I_c^\nu}(t)
\frac{\sqrt{t}J_\nu'(t)J_\nu(r)\sqrt{r}}{t-r}
 \chi_{I_c^\nu}(t).
\]
is a direct consequence of Proposition~\ref{kj} for the following reasons:
\begin{eqnarray*}
 \int_0^\infty
\Big(
\sum_\ell \big|K_\nu^c g_\ell(t)\big|^2
\Big)^{p/2}t^\alpha\; dt
&\lesssim&
\sum_{n=0}^\infty
2^{n\alpha}
\int_{2^{n-1}}^{2^{n+1}}
\Big(
\sum_{\ell:\nu\sim 2^n} 
\big|K_\nu^c g_\ell(t)\big|^2
\Big)^{p/2}\; dt
\\
&\lesssim&
\sum_{n=0}^\infty
2^{n\alpha}
\int_{0}^{\infty}
\Big(
\sum_{\ell:\nu\sim 2^n} 
\big|g_\ell(t)\big|^2
\chi_{I_c^\nu}(t)
\Big)^{p/2}\; dt
\\
&\lesssim&
\sum_{n=0}^\infty
\int_{2^{n-1}}^{2^{n+1}}
\Big(
\sum_{\ell:\nu\sim 2^n} 
\big|g_\ell(t)\big|^2
\Big)^{p/2}t^\alpha\; dt
\\
&\lesssim&
\int_{0}^{\infty}
\Big(
\sum_{\ell:\nu\sim 2^n} 
\big|g_\ell(t)\big|^2
\Big)^{p/2}t^\alpha\; dt,
\end{eqnarray*}
and the result follows taking $\alpha = (n-1)(1-p/2)$.

The remainder terms can be controlled similarly, except for one of them which needs extra arguments, namely
\[
\chi_{I_\infty^\nu}(t)
K_\nu(t,r)
 \chi_{I_c^\nu}(t).
\]

The case $2\le p<2n/(n-1)$ is easy because then $|x|^{(n-1)(1-p/2)}\in A_{p/s}$ for $4/3<s<p$ and we have the estimate
\begin{eqnarray*}
&& \int_0^\infty
\Big(
\sum_\ell \big|K_\nu g_\ell(x)\big|^2
\Big)^{p/2}x^{(n-1)(1-p/2)}\; dx
\\
&&\qquad\ \qquad
\lesssim
 \int_0^\infty
\Big(
\sum_\ell \big|M( g_\ell^s)(x)\big|^{2/s}
\Big)^{p/2}x^{(n-1)(1-p/2)}\; dx
\\
&&\qquad\ \qquad
\lesssim
 \int_0^\infty
\Big(
\sum_\ell \big| g_\ell(x)\big|^{2}
\Big)^{p/2}x^{(n-1)(1-p/2)}\; dx.
\end{eqnarray*}
To treat the case $2n/(n+1)<p\le 2$ let us observe that for $t\ge 3\nu$ we have:
\begin{eqnarray*}
\Big|
\int_0^\infty
K_\nu^1(t,r)g_\ell(r)\chi_{I_c^\nu}(r)\; dr
\Big|
&\lesssim&
\frac 1t\int_{I_c^\nu}\sqrt{r}\big|J_\nu(r)\big|
|g_\ell(r)|\; dr
\\
&\lesssim&
\frac 1t\Big(\int_{I_c^\nu}
|g_\ell(r)|^s\; dr\Big)^{1/s}
\nu^{1/s'}
\\
&\lesssim&
\big(\frac{\nu}{t}\big)^{1/s'}
\big(Mg_\ell^s(t)\big)^{1/s}
\quad\text{ for }s>4/3. 
\end{eqnarray*}
Then we get the integral
\[
 \int_0^\infty 
\Big(\sum_\ell
\big(\frac{\nu}{t}\big)^{2/s'}
\big(Mg_\ell^s(t)\big)^{2/s}
\Big)^{p/2}
t^{(n-1)(1-p/2)}\; dt
\]
and since
$-1<-p/s'+(n-1)(1-p/2)<p/s-1$ it happens that 
$|t|^{-p/s'+(n-1)(1-p/2)}\in A_{p/s}$
allowing us to obtain the estimate needed to finish the proof:
\begin{eqnarray*}
&& 
\int_0^\infty 
\Big(\sum_\ell
\big(\frac{\nu}{t}\big)^{2/s'}
\big(Mg_\ell^s(t)\big)^{2/s}
\Big)^{p/2}
t^{(n-1)(1-p/2)}\; dt
\\
&&\qquad\ \qquad
\lesssim
\int_0^\infty 
\Big(\sum_\ell
\nu^{2/s'}\big|g_\ell(t)\big|^2
\chi_{I_c^\nu}(t)
\Big)^{p/2}
t^{-p/s'+(n-1)(1-p/2)}\; dt
\\
&&\qquad\ \qquad
\lesssim
\int_0^\infty 
\Big(\sum_\ell
\big|g_\ell(t)\big|^2
\Big)^{p/2}
t^{(n-1)(1-p/2)}\; dt.
\end{eqnarray*}
\end{proof}

With similar methods to those developed above, one can obtain the following restriction theorem due to L. Vega \cite{vega}.

\begin{theorem}\label{restriction}
In $\R^n$, $n\ge 2$, given $1\le p<2n/(n+1)$ there exists a finite constant $C_{p,n}$ such that
\[
\|\widehat{f}\,\|_{L^2(S^{n-1})}
\le C_{p,n}
\|{f}\|_{ L^p_{\text{\rm rad}}L^2_{\text{\rm ang}}(\R^n)}
\]
for every rapidly decreasing smooth function $f$.
\end{theorem}

Note that the interval $1\le p<2n/(n+1)$ is bigger than $1\le p<2(n+1)/(n+3)$, which corresponds to the Stein-Tomas restriction theorem (\cite{8}, \cite{9}),
\[
 \|\widehat{f}\,\|_{L^2(S^{n-1})}
\lesssim
 \|{f}\|_{ L^p(\R^n) }.
\]

The proof of the restriction estimate above will be obtained by duality from the corresponding extension estimate
\[
\|\widehat{f d\sigma}\|_{ L^q_{\text{\rm rad}}L^2_{\text{\rm ang}}(\R^n)}
\lesssim_{q,n}
\|{f}\|_{L^2(S^{n-1})}
\qquad\text{for }
q>\frac{2n}{n-1}.
\]
Taking $f(\theta)\equiv 1$ it is easy to check that the estimate above cannot hold when $q\le 2n/(n-1)$, which amounts to show that the range $1\le p<2n/(n+1)$ is sharp in the statement of the theorem.

\

Given $f\in L^2(S^{n-1})$ we have an expansion $f=\sum a_kY_k$ where $Y_k$ is a spherical harmonic of degree $k$, normalized as $\|Y_k\|_{ L^2(S^{n-1})}=1$. Then we can invoke the formula (\cite{16}, \cite{22})
\begin{equation}\label{nhp}
 \widehat{Y_k\; d\sigma}(\xi)=
2\pi i^{k}|\xi|^{1-n/2}
J_{k-1+\frac{n}{2}}(2\pi |\xi|)Y_k\Big(\frac{\xi}{|\xi|}\Big).
\end{equation}

\begin{proof}[Proof of Theorem~\ref{restriction}.]
As we have mentioned, any $f\in L^2(S^{n-1})$ can be written in the form $f=\sum a_kY_k$ where $Y_k$ is a normalized harmonic polynomial of degree $k$ and by \eqref{nhp} it suffices to prove the following inequality:
\[
 \bigg(
\int_0^\infty
\Big(\sum_k |a_k|^2\big|J_{k-1+\frac{n}{2}}(r)\big|^2
\Big)^{q/2}
r^{(1-n/2)q+n-1}\; dr
\bigg)^{1/q}
\lesssim_q\big(\sum |a_k|^2\big)^{1/2},
\]
for every $q>2n/(n-1)$. 

Since $\big|J_\nu(r)\big|\lesssim r^\nu$ when $0\le r\le 1$, the contribution of that interval to the integral above trivializes. 
Let us then consider for a fixed $M=2^m$ ($m=0,1,\dots$) the integral $I_M$ as above but restricted to $M\le r\le 2M$.

We shall first consider the case $n=2$. Then we have

\begin{eqnarray*}
 I_M&=&\int_M^{2M}
\Big(\sum_k |a_k|^2\big|J_{k}(r)\big|^2
\Big)^{q/2}
r\;dr
=
\int_M^{2M}
\Big(\sum_{k< M/2} 
\Big)^{q/2}
r\;dr
\\
& &+\int_M^{2M}
\Big(\sum_{M/2\le k\le 4M} 
\Big)^{q/2}
r\;dr
+\int_M^{2M}
\Big(\sum_{k> 4M} 
\Big)^{q/2}
r\;dr
\\
&=&I_M^1+I_M^2+I_M^3.
\end{eqnarray*}
To estimate $I_M^1$ we observe that $\big|J_k(r)\big|\lesssim 1/r^{1/2}$ uniformly in $k$ when $k\le M/2\le r/2$. Therefore

\[
 I_M^1\lesssim 
\int_M^{2M}
\big(\sum_{k<M/2} |a_k|^2\big)^{q/2}r^{-q/2+1}\; dr
\lesssim
M^{(4-q)/2}\big(\sum |a_k|^2\big)^{q/2}.
\]
Regarding $I_M^3$, we have $k> 4M\ge 2r$ and $|J_k(r)|\lesssim 1/k$ which yields
\[
 I_M^3\lesssim 
\int_M^{2M}
\big(\sum_{k>4M} k^{-2}|a_k|^2\big)^{q/2}r\; dr
\lesssim
M^{2-q}\big(\sum |a_k|^2\big)^{q/2}.
\]
Finally let us write
\[
 I_M^2\lesssim 
\sum_\alpha\int_{G_\alpha}
\big(\sum_{M/2\le k\le 4M}|a_k|^2|J_k(r)|^2\big)^{q/2}r\; dr
\]
where
\[
 G_\alpha=\big[\frac{M}{2}+\alpha M^{1/3},\frac{M}{2}+(\alpha+1) M^{1/3}\big]
\]
is an interval in the real line and $\alpha=0,1,2,\dots,\big[\frac 32 M^{2/3}\big]$.

Let us also define
\[
 A_\alpha =\sum_{k\in G_\alpha}|a_k|^2.
\]
Then we have
\begin{eqnarray*}
I_M^2&\le&
\sum_\beta\int_{G_\beta}
\bigg(\sum_{\alpha\le \beta}
A_\alpha
\Big(
\frac{1}{(|\alpha-\beta|+1)^{1/4}M^{1/3}}
\Big)^2\bigg)^{q/2}rdr
\\
&& +
\sum_\beta\int_{G_\beta}
\bigg(\sum_{\alpha\geq \beta}
A_\alpha
\Big(
\frac{1}{(|\alpha-\beta|+1)M^{1/3}}
\Big)^2\bigg)^{q/2}rdr
= I_M^{2,1}+I_M^{2,2}.
\end{eqnarray*}
And
\begin{eqnarray*}
I_M^{2,1}&\lesssim&
\sum_{\beta\le \frac 32 M^{2/3}}
\bigg(\sum_{\alpha\le \beta}
\frac{A_\alpha}{(|\alpha-\beta|+1)^{1/2}}
\bigg)^{q/2}
M^{(4-q)/3}
\lesssim
M^{(4-q)/3}
\Big(
\sum_\gamma A_\gamma^s
\Big)^{q/2s}
\\
&\lesssim&
M^{(4-q)/3}
\Big(
\sum_\gamma A_\gamma
\Big)^{q/2}
\lesssim
M^{(4-q)/3}
\Big(
\sum_k |a_k|^2
\Big)^{q/2}
\end{eqnarray*}
where we have taken $2/q=1/s-1/2$ and used the fact that $q>4$. 

The term $I_M^{2,2}$ is controlled by the same argument. Therefore, when $n=2$, $q>4$, adding all the estimates above over the dyadic intervals $[M,2M]$, $M=2^m$, we get 
\[
 \sum_m
\big(
2^{m(4-q)/2}
+
2^{m(4-q)/3}
\big)
\Big(
\sum_k |a_k|^2
\Big)^{q/2}
\lesssim
\Big(
\sum_k |a_k|^2
\Big)^{q/2}
\]
and the theorem is proved in this case.

\

In the general case $n\ge 2$, we have to estimate
\[
\int_{2^m}^{2^{m+1}}
\Big(\sum_k |a_k|^2\big|J_{k-1+\frac{n}{2}}(r)\big|^2
\Big)^{q/2}
r^{(1-n/2)q+n-1}\; dr
\bigg)^{1/q},
\]
that with the previous argument is
\[
\lesssim
2^{m(n-2)(-q/2+1)}\big(
2^{m(4-q)/2}
+
2^{m(4-q)/3}
\big)
\Big(
\sum_k |a_k|^2
\Big)^{q/2}
\]
and it is enough to note that the resulting exponent in the powers of 2 is negative for $q>2n/(n-1)$.
\end{proof}

\section{Appendix: Remarks on $A_p$ weights}

Throughout this paper we have taken advantage of the fact that 
\[
 \big(M\omega^s(x)\big)^{1/s},
\qquad 1<s<\infty,
\]
is a weight in the class $A_1$, for every integrable function $\omega$, whose $A_p$ bounds are estimated independently of $\omega$.

This property was discovered (see \cite{13}) with the disc multiplier problem in mind, as an efficient manner of relating functionally the boundedness properties of singular integrals to maximal functions involving different directions of $\R^n$.

To our knowledge, the family $\big(M\omega^s(x)\big)^{1/s}$ is the more extended class of known weights; the other known examples given by powers $|x|^\alpha$ which have also played an important role in several proofs of this paper. It is then interesting to realize that those power weights  $|x|^\alpha$ can also be considered particular cases of the construction $\big(M\omega^s(x)\big)^{1/s}$. In the following we will present the details in dimension one, leaving the general case as an exercise. 

\

\begin{lemma}
\[
 M\Big(\big(M\omega^s\big)^{1/s}\Big)(x)
\lesssim_s
 \big(M\omega^s(x)\big)^{1/s}.
\]
\end{lemma}
\begin{proof}
Let $Q$ be a cube containing the point $x$ and let us denote by $M_{Q^*}$ the maximal operator \lq\lq restricted'' to subcubes of $Q^*$ (the double of $Q$).
It is well-known
 that the mapping $f\mapsto M_{Q^*}f$ is bounded from $L^1(Q^*)$ to $L^p(Q^*)$, $p<1$.
We have the splitting
\[
 Q=E\cup(Q\setminus E)
\]
where
\[
 M\omega^s\big|_E
=
 M_{Q^*}\omega^s\big|_E,
\qquad
 M\omega^s\big|_{Q\setminus E}
\approx \text{constant},
\]
i.e. In $Q\setminus E$, $\max M\omega^s(x)\le C\min M\omega^s(x)$ for some universal constant $C$.
Therefore
\begin{eqnarray*}
&& \frac{1}{|Q|}
\int_Q
\big(
M\omega^s(y)
\big)^{1/s}\; dy
\\
&&
\qquad\ \qquad
=
 \frac{1}{|Q|}
\int_E
\big(
M_{Q^*}\omega^s(y)
\big)^{1/s}\; dy
+
\frac{1}{|Q|}
\int_{Q\setminus E}
\big(
M\omega^s(y)
\big)^{1/s}\; dy
\\
&&
\qquad\ \qquad
\lesssim_s
\Big(
 \frac{1}{|Q|}
\int_{Q^*}
\omega^s(y)\; dy
\Big)^{1/s}
+
\sup_{y\in Q\setminus E}
\big(
M\omega^s(y)
\big)^{1/s}
\\
&&
\qquad\ \qquad
\lesssim_s
\big(
M\omega^s(x)
\big)^{1/s}
+
\inf_{y\in Q\setminus E}
\big(
M\omega^s(y)
\big)^{1/s}
\end{eqnarray*}
which is $\lesssim_s
\big(
M\omega^s(x)
\big)^{1/s}$.
\end{proof}

\begin{corollary}
 $|x|^\alpha$ is in the class $A_p$ if and only if $-1<\alpha<p-1$.
\end{corollary}
\begin{proof}
In the interval $-1<\alpha\le 0$ we just observe that $f(x)=|x|^\alpha$ satisfies
\[
\Big(
\frac{1}{1+s\alpha}
\Big)^{1/s}|x|^\alpha\le
 \big(Mf^s(x)\big)^{1/s}
\le
\Big(
\frac{2}{|x|}
\int_0^{|x|}
t^{\alpha s}\; dt
\Big)^{1/s}
\le
\Big(
\frac{2}{1+s\alpha}
\Big)^{1/s}|x|^\alpha
\]
under
the hypothesis that $s\alpha>-1$. Thus, from the previous lemma, we see that the weight is in $A_1\subset A_p$.

For $0\le\alpha<p-1$, we use the characterization of $A_p$ in terms of the boundedness of the maximal function or the Hilbert transform. We have
\[
 \Big(
\int \big|Hf(x)\big|^p |x|^\alpha\; dx
\big)^{1/p}
=
\sup
\int Hf(x)g(x)|x|^\alpha\; dx
\]
where the supremum is taken over all $g$ such that 
\[
 \int \big|g(x)\big|^q|x|^\alpha\; dx\le 1,
\qquad\frac 1p+\frac 1q=1.
\]
But we also have
\begin{equation}\label{wap}
 \Big|
\int Hf(x)g(x)|x|^\alpha\; dx
 \Big|
= 
\Big|
\int f(x)H\big(g(x)|x|^\alpha\big)\; dx
 \Big|.
\end{equation}
This quantity is bounded by
\[
 \Big(
\int |f(x)|^p|x|^\alpha\; dx
 \Big)^{1/p}
 \Big(
\int \big|H(g(x)|x|^\alpha)\big|^q|x|^{-\alpha q/p}\; dx
 \Big)^{1/q}.
\]
Then we observe
\[
 0>-\alpha\frac qp=-\alpha\frac 1p \frac{p}{p-1}>-1.
\]
Therefore, using the previous case,  \eqref{wap} is 
\[
 \lesssim
 \Big(
\int |f(x)|^p|x|^\alpha\; dx
 \Big)^{1/p}
 \Big(
\int \big|g(x)\big|^q|x|^\alpha\; dx
 \Big)^{1/q},
\]
and we are done.
\end{proof}


\



\begin{thebibliography}{99}

\bibitem{10}  A. Carbery, E. Hern\'andez\ and\ F. Soria, The behaviour on radial functions of maximal operators along arbitrary directions and the Kakeya maximal operator, Tohoku Math. J. (2) {\bf 41} (1989), no.~4, 647--656.

\bibitem{17} A. Carbery, E. Romera\ and\ F. Soria, Radial weights and mixed norm inequalities for the disc multiplier, J. Funct. Anal. {\bf 109} (1992), no.~1, 52--75.

  	         
\bibitem{2} L. Carleson\ and\ P. Sj\"olin, Oscillatory integrals and a multiplier problem for the disc, Studia Math. {\bf 44} (1972), 287--299. (errata insert).

\bibitem{11} A. C\'ordoba, The {K}akeya maximal function and the spherical summation
              operators, PhD thesis,  University of Chicago, 1974.
    
\bibitem{12} A. C\'ordoba, Maximal functions, covering lemmas and Fourier multipliers, in {\it Harmonic analysis in Euclidean spaces (Proc. Sympos. Pure Math., Williams Coll., Williamstown, Mass., 1978), Part 1}, 29--50, Proc. Sympos. Pure Math., XXXV, Part Amer. Math. Soc., Providence, RI.

\bibitem{4} A. C\'ordoba, A note on Bochner-Riesz operators, Duke Math. J. {\bf 46} (1979), no.~3, 505--511.


\bibitem{14} A. C\'ordoba, Translation invariant operators, in {\it Fourier analysis (Proc. Sem., El Escorial, 1979)}, 117--176, Asoc. Mat. Espa\ nola, 1 Asoc. Mat. Espa\~nola, Madrid.

\bibitem{6} A. C\'ordoba, The disc multiplier, Duke Math. J. {\bf 58} (1989), no.~1, 21--29.


\bibitem{13} A. Cordoba\ and\ C. Fefferman, A weighted norm inequality for singular integrals, Studia Math. {\bf 57} (1976), no.~1, 97--101. 

\bibitem{15} K. M. Davis\ and\ Y.-C. Chang, {\it Lectures on Bochner-Riesz means}, London Mathematical Society Lecture Note Series, 114, Cambridge Univ. Press, Cambridge, 1987.

\bibitem{5} C. Fefferman, The multiplier problem for the ball, Ann. of Math. (2) {\bf 94} (1971), 330--336.

\bibitem{3} C. Fefferman, A note on spherical summation multipliers, Israel J. Math. {\bf 15} (1973), 44--52.




\bibitem{20} 
J. Garc\'\i a-Cuerva\ and\ J. L. Rubio de Francia, {\it Weighted norm inequalities and related topics}, North-Holland Mathematics Studies, 116, North-Holland, Amsterdam, 1985.

\bibitem{7} G. Mockenhaupt, On radial weights for the spherical summation operator, J. Funct. Anal. {\bf 91} (1990), no.~1, 174--181. 

\bibitem{19} B. Muckenhoupt, Weighted norm inequalities for the Hardy maximal function, Trans. Amer. Math. Soc. {\bf 165} (1972), 207--226. 

\bibitem{18} E. T. Sawyer, A characterization of a two-weight norm inequality for maximal operators, Studia Math. {\bf 75} (1982), no.~1, 1--11. 

\bibitem{9} E. M. Stein, {\it Harmonic analysis: real-variable methods, orthogonality, and oscillatory integrals}, Princeton Mathematical Series, 43, Princeton Univ. Press, Princeton, NJ, 1993.

\bibitem{21} E. M. Stein\ and\ S. Wainger, Maximal functions associated to smooth curves, Proc. Nat. Acad. Sci. U.S.A. {\bf 73} (1976), no.~12, 4295--4296.

\bibitem{22} E. M. Stein\ and\ G. Weiss, {\it Introduction to Fourier analysis on Euclidean spaces}, Princeton Univ. Press, Princeton, NJ, 1971. 



\bibitem{8} P. A. Tomas, A restriction theorem for the Fourier transform, Bull. Amer. Math. Soc. {\bf 81} (1975), 477--478. 

\bibitem{vega} L. Vega, El multiplicador de Schr\"{o}dinger. La funcion maximal y los operadores
 de restricci\'{o}n, Universidad Aut\'{o}noma de Madrid, PhD thesis (1988).


\bibitem{16} G. N. Watson, {\it A treatise on the theory of Bessel functions}, reprint of the second (1944) edition, Cambridge Mathematical Library, Cambridge Univ. Press, Cambridge, 1995. 

\bibitem{23} T. Wolff, An improved bound for Kakeya type maximal functions, Rev. Mat. Iberoamericana {\bf 11} (1995), no.~3, 651--674. 

\bibitem{1} A. Zygmund, {\it Trigonometric series. Vol. I, II}, third edition, Cambridge Mathematical Library, Cambridge Univ. Press, Cambridge, 2002. 


\end{thebibliography}
\end{document}